\documentclass[12pt]{article}

\usepackage{amsmath, amsfonts,amssymb,mathrsfs,amscd,comment}
\usepackage{comment}
\usepackage{color}
\usepackage[mathscr]{eucal}
\usepackage[math]{easyeqn}
\usepackage{etoolbox}
\usepackage{amssymb}
\usepackage[table]{xcolor}
\usepackage{pgfplots}
\usepackage{multirow}
\usepackage{listings}
\usepackage{emptypage}
\usepackage{natbib}
\usepackage{chngcntr}
\usepackage{sidecap}
\usepackage{caption}
\usepackage{pgfplots}
\usepackage{hyperref}
\definecolor{darkgreen}{rgb}{0,0.6,0}
\definecolor{darkblue}{rgb}{0,0,0.6}
\definecolor{orange}{RGB}{119,34,51}
\hypersetup{linktocpage,
            colorlinks,
            linkcolor=darkgreen,
            linktoc=blue,
            citecolor=darkblue,
            urlcolor=black,
            filecolor=black
            }

\usepackage{nameref}
\usepackage{url}

\usepackage{amsthm}

\newtheorem{theorem}{Theorem}[section]
\numberwithin{equation}{section}
\newcounter{example}[section]
\numberwithin{example}{section}
\newtheorem{proposition}[theorem]{Proposition}
\newtheorem{lemma}[theorem]{Lemma}

\newtheorem{definition}[theorem]{Definition}
\newtheorem{exmp}[example]{Example}

\newtheorem{remark}[theorem]{Remark}
\newenvironment{example}{\begin{exmp}\rm}{\end{exmp}}

\def\cond{\, \big| \,}

\def\ex{\mathrm{e}}

\def\K{K}

\def\Cs{E}

\def\eps{\epsilon}

\def\B{\cc{B}}

\def\N{\mathcal{N}}

\def\F{\mathcal{F}}

\def\P{{\mathbb P}}
\def\E{{\mathbb E}}

\def\cond{\, \big| \,}

\def\ex{\mathrm{e}}

\def\K{K}

\def\Cs{E}

\def\eps{\epsilon}

\def\B{\cc{B}}

\def\N{\mathcal{N}}

\renewcommand{\(}{$\,}
\renewcommand{\)}{\,$}

\def\eqdef{\stackrel{\operatorname{def}}{=}}

\renewcommand{\bar}[1]{\overline{#1}}
\renewcommand{\hat}[1]{\widehat{#1}}
\renewcommand{\tilde}[1]{\widetilde{#1}}

\renewcommand{\Gamma}{\varGamma}
\renewcommand{\Pi}{\varPi}
\renewcommand{\Sigma}{\varSigma}
\renewcommand{\Delta}{\varDelta}
\renewcommand{\Lambda}{\varLambda}
\renewcommand{\Psi}{\varPsi}
\renewcommand{\Phi}{\varPhi}
\renewcommand{\Theta}{\varTheta}
\renewcommand{\Omega}{\varOmega}
\renewcommand{\Xi}{\varXi}
\renewcommand{\Upsilon}{\varUpsilon}

\def\Var{\operatorname{Var}}

\def\argmin{\operatornamewithlimits{argmin}}

\def\R{I\!\!R}

\def\kappa{\varkappa}

\def\B{\mathfrak{B}}

\def\C{\mathbf{C}}
\def\Ch{\hat{C}}
\def\Cs{\mathcal{C}}

\def\B{\mathfrak{B}}
\def\Q{\mathbb{Q}}
\def\R{\mathbb{R}}

\def\Kh{\mathcal{K}}
\def\K{\mathbf{K}}

\numberwithin{equation}{section}

\swapnumbers

\RequirePackage[normalem]{ulem}
\RequirePackage{color}\definecolor{RED}{rgb}{1,0,0}\definecolor{BLUE}{rgb}{0,0,1}

 \DeclareMathOperator{\Exp}{Exp}
 \DeclareMathOperator{\Poiss}{Poiss}

\renewcommand{\eps}{\varepsilon}
\renewcommand{\phi}{\varphi}

\renewcommand{\theta}{\vartheta}
\renewcommand{\subset}{\subseteq}

\renewcommand{\abs}[1]{\lvert #1 \rvert}

\renewcommand{\le}{\leqslant}
\renewcommand{\ge}{\geqslant}

\allowdisplaybreaks[1]
\begin{document}



\title{Unbiased estimation of the volume\\ of a convex body}
\author{ \parbox{7cm}{\centering Nikolay Baldin\\\small Institute of Mathematics\\ Humboldt-Universit{\"a}t zu Berlin\\baldin.np@gmail.com} \hfill \parbox{7cm}{\centering Markus Rei\ss\footnote{We are grateful for helpful comments and questions by Mathias Reitzner and the anynonymous referee. Financial support by the DFG via Research Unit FOR1735 {\it Structural Inference in Statistics} is gratefully acknowledged.} \\\small Institute of Mathematics\\ Humboldt-Universit{\"a}t zu Berlin\\mreiss@math.hu-berlin.de}}



\maketitle

\begin{abstract}
Based on observations of points uniformly distributed over a convex set  in $\R^d$,
a new estimator for the volume of the convex set is proposed.
The estimator is minimax optimal and also efficient non-asymptotically:
it is nearly unbiased with minimal variance among
 all unbiased oracle-type estimators. Our approach is based on a Poisson point process model and
as an ingredient, we prove that the convex hull is a sufficient and complete statistic.
No hypotheses on the boundary of the convex set are imposed.  In a numerical study, we show that the estimator outperforms earlier estimators
for the volume. In addition, an adjusted set estimator for the convex body itself is proposed.
\end{abstract}







\section{Introduction}

Driven by  applications in image analysis and signal processing, the estimation of the support of a density
attracts a lot of statistical activity.
In many cases it is  natural to assume a convex shape for the support set. First fundamental  results for convex support estimation have been achieved by \cite{korostelev1993minimax, korostelev1994asymptotic} who prove minimax-optimal rates of convergence in Hausdorff distance for a set estimator. In particular, \cite{korostelev1993minimax} prove that the convex
hull of the points \(\Ch_n \), which is a maximum likelihood estimator for the set \(C\),
is rate-optimal. Interestingly,  the volume \( |\Ch_n|
\) of the convex hull is not rate-optimal for estimation of the volume \(|C|\) of the convex set and
an alternative two-step estimator, optimal up to a logarithmic factor, was proposed. A fully rate-optimal estimator for the volume of a convex set with smooth boundary was then
constructed by \cite{gayraud1997} based on three-fold sample splitting. For various extensions and applications of convex support estimation, let us refer to \cite{mammen1995asymptotical, guntuboyina2012optimal,brunel:tel-01066977} and the literature cited there. Related ideas under H\"older and monotonicity constraints, respectively, have been adopted by \cite{ReissSelk14} for a one-sided regression model.

Our contribution is the construction of a very simple volume estimator which is not only rate-optimal over all convex sets without boundary restrictions, but even adaptive in the sense that it attains almost the parametric rate if the convex set is a polytope. Our approach is non-asymptotic and provides much more precise properties.
The analysis is based on a Poisson point process (PPP) observation model with intensity $\lambda>0$ on the convex set $C\subset\R^d$. We thus observe
\begin{EQA}[c]
X_1, ..., X_N \overset{i.i.d.}{\sim} U(C), \, \quad
N \sim \text{Poiss}(\lambda\abs{C}),
\end{EQA}
where \((X_n), N\) are independent, see Section~\ref{theoretical_digression_section} below for a concise introduction to the PPP model. Using Poissonisation and de-Poissonisation techniques, this model exhibits  asymptotic properties like the uniform model,
i.e. a sample of $n=\lambda\abs{C}$ uniformly on $C$ distributed random variables $X_1,\ldots,X_n$. The beautiful geometry of the PPP model, however, allows for much more concise ideas and proofs, see also \cite{meister2013asymptotic} for connections between PPP and regression models with irregular error distributions. From an applied perspective, PPP models  are often natural, e.g. for spatial count data of photons or other emissions.

For known intensity $\lambda$ of the PPP, we construct in Section~\ref{oracle_case_section}  an {\it oracle} estimator $\hat\theta_{oracle}$. Theorem~\ref{ThmOracle} shows that this estimator is UMVU (uniformly of minimum variance among unbiased estimators) and rate-optimal. To this end,  moment bounds from stochastic geometry for the missing volume of the convex hull,  obtained by \cite{barany1988convex} and  \cite{dwyer1988convex} are essential. Moreover, we derive results of independent interest:
the convex hull $\Ch=\text{conv}\{X_1,\ldots,X_N\}$ forms a sufficient and complete statistic (Proposition~\ref{PropSuffCompl}) and
 the Poisson point process, conditionally on \(  \Ch\), remains Poisson within its convex hull (Theorem~\ref{LemMeasCond}).

 For the more realistic case of unknown intensity $\lambda$, we analyse in Section~\ref{unknown_intensity_section} our final estimator
\begin{EQA}[c]
\hat{\theta} \eqdef \frac{N + 1}{N_\circ + 1} |\Ch| \,,
\end{EQA}
where $N_\circ$ denotes the number of observed points in the interior of $\Ch$.
We are able to prove a sharp oracle inequality, comparing the risk of this estimator to that of $\hat\theta_{oracle}$. Here, very recent and advanced results  by \cite{reitzner2003random,pardon2011,reitzner2015} on the variance of the number of points $N_{\partial}$ on the boundary of $\Ch$ and the missing volume $\abs{C\setminus\Ch}$  are of key importance. This fascinating interplay between stochastic geometry and statistics prevails throughout the work.

The lower bound showing that $\hat\theta$ is indeed minimax-optimal is proved in Theorem~\ref{lower_bound_theorem} by
adopting the proof of the lower bound in the uniform model by \cite{gayraud1997}. A small simulation study is presented in Section~\ref{numerical_study_section}. Moreover, we propose to enlarge the convex hull set by the factor $((N + 1)/(N_\circ + 1))^{1/d}$ and we study its error as an estimator of the set $C$ itself. The proof of Lemma~\ref{lemma_bias_plug_in} is deferred to the Appendix.


\section{Digression on Poisson Point Processes}
\label{theoretical_digression_section}

Most of the results and notation are adapted from \cite{karr1991point}.
We fix a compact convex set $\mathbf E$ in $\R^d$ with non-empty interior as a state space and denote by
\(\mathcal{E}\) its Borel \(\sigma\)-algebra.  We define
the  family of convex subsets \(\C = \{C \subseteq \mathbf{E}, \text{convex}, \text{closed} \}\)
(this implies that all sets in \(\C\) are compact) and the family of
compact subsets \(\K = \{K \subseteq \mathbf{E}, \text{compact} \}\). It is natural to equip the space \(\C\) (resp. \(\K\)) with the Hausdorff-metric $d_H$ and its
 Borel \(\sigma\)-algebra $\B_{\C}$ (resp. \(\B_{\K}\)). Then $(\C,d_H)$ is a compact and thus separable  space and the mapping $(x_1,\ldots,x_k)\mapsto \text{conv}\{x_1,\ldots,x_k\}$, which generates the convex hull of points $x_i\in{\mathbf E}$, is continuous from ${\mathbf E}^k$ to $(\C,d_H)$. 

On \((\mathbf{E}, \mathcal{E})\) we define the set of point measures
 \(\mathbf{M} = \{m \text{ measure on } \mathcal{E}\,:\,m(A) \in \mathbb{N},\,\, \forall A \in \mathcal{E}\}\) equipped
with the \(\sigma\)-algebra \(\mathcal{M} = \sigma(m\mapsto m(A), A \in\mathcal{E} )\).
Let \(C_c^{+}(\mathbf{E})\) be the collection of continuous functions \(\mathbf{E} \mapsto [0,\infty)\) with compact support.
A useful topology for \(\mathbf{M}\) is the \emph{vague topology} which makes \(\mathbf{M}\) a complete, separable metric space, cf.
Section 3.4 in
\cite{resnick2013extreme}.
A sequence of point measures \(m_n \in \mathbf{M}\) then converges vaguely to a limit \(m \in \mathbf{M}\) if and only if
\(m_n[f] \to m[f]\) for all \(f \in C_c^{+}(\mathbf{E}) \) where \(m[f] = \int_{\mathbf{E}} f d m \).
Let \(( \Omega, \F, \P)\) be an abstract probability space.
We call a measurable mapping
\( \N: \Omega \to \mathbf{M}\)
a Poisson point process (PPP) of intensity $\lambda>0$ on \(C\in\mathbf{C}\) if
\begin{itemize}
\item for any \(A \in  \mathcal{E}\), we have \(\N(A) \sim \text{Poiss}\bigl(\lambda\abs{A\cap C}\bigr)\), where $\abs{A\cap C}$ denotes the Lebesgue measure of $A\cap C$;
\item for all mutually disjoint sets \(A_1,..., A_n \in    \mathcal{E}\), the
random variables \(\N(A_1),...,\N(A_n) \) are independent.
\end{itemize}
For statistical inference, we assume the Poisson point process to be defined on a set of non zero  Lebesgue measure, i.e. \(|C| >0\).
A more constructive and intuitive representation of the PPP  \(\N\)
is $\N = \sum_{i = 1}^{N} \delta_{X_i}$
for \(N \sim \text{Poiss}(\lambda\abs{C})\) and  i.i.d. random variables \((X_i)\),
independent of \(N\) and distributed uniformly  \(\P(X_i \in A) = |A \cap C| / |C|\),
so that  \(\N(A)  =  \sum_{i = 1}^{N} {\bf 1}( X_i \in A )\)  for any \(A \in  \mathcal{E}\).

We  consider the convex hull of the PPP points  \(\hat C : \mathbf{M } \to \mathbf{C}\) defined by
 \(\hat C(\N) : =\text{conv}\{X_1, ..., X_N\}\), which by the above continuity property of the convex hull is a random element with values in the Polish space $(\C,d_H)$, see also \cite{davis1987convex} for a detailed study of the continuity of the convex hull. For a short notation, we shall further write \(\Ch\)
to denote the convex hull of the process \(\N\).
In the sequel, conditional expectations and probabilities with respect to $\hat C$ are thus well defined.
We can also evaluate the probability
\[
\P_C\bigl(\hat C \in  A\bigr)
	= \sum_{k = 0}^{\infty} \frac{\ex^{-\lambda\abs{C}} \lambda^k}{k!}
									\int_{C^k} {\bf 1}(\text{conv}\{x_1, ..., x_k\} \in A ) d(x_1,..., x_k)
\]
for \(A \in \B_{\C}\). Usually, we only write the subscript $C$ or sometimes $(C,\lambda)$ when different probability distributions are considered simultaneously.
The likelihood function \(\frac{d \P_{C,\lambda}}{d \P_{\mathbf{E},\lambda_0}}\) for \(C \in{\mathbf C}\) and $\lambda,\lambda_0>0$
is then given by
\begin{EQA}
\frac{d \P_{C,\lambda}}{d \P_{\mathbf{E},\lambda_0}}(X_1,...,X_N)
&=& \ex^{\lambda_0|{\mathbf E}|-\lambda|C|}(\lambda/\lambda_0)^N {\bf 1}(\forall\, i=1,...,N:\,X_i \in C)\\
&=& \ex^{\lambda_0|{\mathbf E}|-\lambda|C|}(\lambda/\lambda_0)^N {\bf 1}(\hat C\subseteq C)\,,
\label{fPPeb}
\end{EQA}
cf. Thm. 1.3 in \cite{Kutoyants1998}. For the last line, we have used that a point set is in $C$ if and only if its convex hull is contained in $C$.

For the set-indexed process $(\N(K),K\in\K)$ we define 
its natural set-indexed  filtration
\begin{EQA}[c]
\F_K \eqdef \sigma(\{\N(U); U \subseteq K,  U \in \K \})
\end{EQA}
for any \( K \in \K\).
The filtration \((\F_K, K \in \K)\) possesses the following properties:
\begin{itemize}
\item monotonicity: \(\F_{K_1} \subseteq \F_{K_2} \) for any \(K_1, K_2 \in \K\) with \(K_1 \subseteq K_2 \),
\item continuity from above: \(\F_K = \cap_{i=1}^{\infty}\F_{K_i}\) if \(K_i \downarrow K \);
\end{itemize}
cf. \cite{zuyev1999stopping}. By construction, the restriction \(\N_K = \N(\cdotp \cap K)\) of the point process \(\N\) onto
\(K \in \mathbf{K}\) is \(\F_K\)-measurable (in fact, \(\F_K = \sigma(\{\N_K(U); U \in \mathbf{K}\})\)). In addition, it can be easily seen that \(\N_K\) is a Poisson point process in \(\mathbf{M}\), cf. the
Restriction Theorem in \cite{kingman1992poisson}, and thus \(\Ch(\N_K) = \text{conv}(\{X_1,\ldots,X_N\}\cap K)\) is by the above arguments \(\F_K\)-measurable.

A random compact set \(\Kh\) is a measurable mapping \(\Kh:  (\mathbf{M},\mathcal{M}) \to (\K,\B_\K)\).
Note that \cite{zuyev1999stopping} defines a random compact set as a measurable mapping from \( (\mathbf{M},\mathcal{M})\) to
\((\K,\sigma_{\K})\) where \(\sigma_{\K}\) is the so-called \emph{Effros} \(\sigma\)-algebra generated by the sets \(\{ F \in \K: F \cap K \neq \emptyset \}\), \(K \in \K\). Thanks to
Thm. 2.7  in \cite{molchanov2006}, the Effros \(\sigma\)-algebra  \(\sigma_{\K}\)  induced on the family of compact sets \(\K\) coincides with
the Borel \(\sigma\)-algebra \(\B_\K\), and we prefer to stick to the first definition of a random compact set for convenience.
Next, we recall the definition of stopping sets from \cite{rozanov1982markov} in complete analogy with stopping times.
\begin{definition}
A random compact set \(\Kh\) is called an \(\F_K\)-stopping set if \(\{\Kh \subseteq K\} \in \F_K\)
for all \(K \in \K\). The sigma-algebra of \(\Kh\)-history is defined as
$\F_{\Kh} = \{A \in \mathfrak{F}: A \cap \{ \Kh \subseteq K \} \in \F_K\,\,\, \forall K \in \K \}$, where
\(\mathfrak{F} =   \sigma(\F_K; K \in \K)\).
\end{definition}

For a set $A\subset {\mathbf E}$ let $A^c$ denote its complement.

\begin{lemma}
The set $\hat\Kh\eqdef\overline{\Ch^c}$, the closure of the complement of the convex hull, is an $(\F_K)$-stopping set.
\end{lemma}

\begin{proof}
We claim  $\hat\Kh\subseteq K$ if and only if $K^c\subseteq \text{conv}(\{X_1,\ldots,X_N\}\cap K)$. Indeed, if $\hat\Kh\subseteq K$ holds, then the boundary $\partial\Ch=\partial\hat\Kh$ is in $K$ which implies $\text{conv}(\{X_1,\ldots,X_N\}\cap K)=\Ch$. Consequently, $K^c\subseteq \hat\Kh^c\subseteq \Ch=\text{conv}(\{X_1,\ldots,X_N\}\cap K)$ holds. Conversely,  $K^c\subseteq \text{conv}(\{X_1,\ldots,X_N\}\cap K)$ implies immediately $K^c\subseteq \Ch$ and thus
$\Ch^c\subseteq K$. Since $K$ is closed, we obtain $\hat\Kh\subseteq K$.

Since $\{X_1,\ldots,X_N\}\cap K$ are the realisations of the point process inside $K$ and the convex hull is measurable, we conclude $\{K^c\subseteq \text{conv}(\{X_1,\ldots,X_N\}\cap K)\}\in\F_K$.
\end{proof}

We shall further use the following short notation:  \(N =\N(C)\) denotes the total number of points, \(N_\circ = \N({\Ch}^\circ)\) the number of points in the interior
of the convex hull \(\Ch\) and \( N_{\partial} = \N(\partial\Ch)=\N(\partial\hat\Kh) \) the number of points on the boundary of the convex hull.
For asymptotic bounds we write \(f(x) = O(g(x))\) or \(f(x) \lesssim g(x)\) if \(f(x)\) is bounded by a constant multiple of \(g(x)\)  and
 \(f(x) \thicksim g(x)\) if \(f(x) \lesssim g(x)\)  as well as \(g(x) \lesssim f(x)\).


\section{Oracle case: intensity \(\lambda\) is known}
\label{oracle_case_section}

For a PPP on $C\in{\mathbf C}$ with intensity $\lambda>0$, we know  $N\sim\Poiss(\lambda\abs{C})$.
In the oracle case, when the intensity \(\lambda \) is known,  $N/\lambda$ estimates $\abs{C}$ without bias and yields the classical parametric rate in $\lambda$:
\begin{EQA}[c]
\E [(N/\lambda - |C|)^2] = \lambda^{-2}\Var(N)  = \frac{|C|}{\lambda}\,.
\end{EQA}
Another natural idea might be to use the plug-in estimator \( |\Ch|\) whose error is given by the missing volume and satisfies
\begin{EQA}[c]
\E [(\abs{\Ch} - |C|)^2] = \E[|C \setminus \Ch|^2] =   O(|C|^{2(d-1)/(d+1)}\lambda^{-4/(d+1)})\,,
\end{EQA}
where the bound is obtained  similarly to \eqref{thm_upper_bound} and \eqref{EcbCh} below. This means that its error is of smaller order than $\lambda^{-1}$ for $d\le 2$, but larger for $d\ge 4$. For any $d\ge 2$, however, both convergence rates are worse than the minimax-optimal rate $\lambda^{-(d+3)/(d+1)}$, established below.

The way to improve these estimators is to observe that by the likelihood representation \eqref{fPPeb} for $\lambda=\lambda_0$ and the Neyman factorisation criterion the convex hull is a sufficient statistic.
Consequently, by the Rao-Blackwell theorem, the
 conditional expectation of \( N / \lambda\) given the convex hull \(\Ch \)
is an estimator with smaller mean squared error (MSE).

The number of points \(N\) can be split into the number \(N_{\partial}\) of points on the boundary
and the number  \(N_\circ \) of points in the interior of the convex hull.
The following theorem is essential in deriving the oracle estimator. Although the statement
of the theorem is quite intuitive and already used in \cite{privault2012invariance}, the proof turns out to be nontrivial and is deferred to the Appendix.

\begin{theorem}
\label{LemMeasCond}
The number \(N_{\partial}\) of points on the boundary of the convex hull  is measurable with respect to the sigma-algebra of
\(\hat\Kh\)-history \(\F_{\hat\Kh}\). The number of points in the interior of the convex hull \(N_\circ\) is, conditionally on \(  \F_{\hat\Kh}\), Poisson-distributed:
\begin{equation}\label{NccCs}
N_\circ \cond \F_{\hat\Kh} \sim \text{Poiss}(\lambda_\circ)\text{ with }\lambda_\circ \eqdef \lambda | \Ch |.
\end{equation}
In addition, we have \(\F_{\hat\Kh} = \sigma(\Ch)\), where
the latter is the sigma-algebra \( \sigma(\{\Ch \subseteq B, B \in \C\})\) completed with the null sets in  \(\mathfrak{F}\).
\end{theorem}

With Theorem~\ref{LemMeasCond} at hand, we obtain the \emph{oracle} estimator
\begin{EQA}[c]
\label{sbvcbxcvb}
\hat\theta_{oracle} \eqdef \E\Bigl[\frac{N }{\lambda} \cond \Ch \Bigr] =   \E\Bigl[\frac{N_\circ + N_{\partial}}{\lambda} \cond \Ch  \Bigr]
= | \Ch |+\frac{ N_{\partial}}{\lambda} \,,
\end{EQA}
where conditioning on \(\Ch\) means conditioning on \(\sigma(\Ch)=\F_{\hat\Kh}\).

\begin{theorem}\label{ThmOracle}
For known intensity $\lambda > 0$, the oracle estimator \(\hat\theta_{oracle}\) is unbiased and of minimal variance among all unbiased estimators (UMVU). It satisfies
\[ \Var(\hat\theta_{oracle})= \frac{1}{\lambda} \E[ |C \setminus \Ch|]\,.\]
Its worst case mean squared error over $\C$ decays as $\lambda\uparrow \infty$  like $\lambda^{-(d+3)/(d+1)}$ in dimension $d$:
\begin{EQA}[c]
\limsup_{\lambda\to \infty} \lambda^{(d+3)/(d+1)}\sup_{C \in \C, |C| >0} \Big\{\abs{C}^{-(d-1)/(d+1)} \E\big[(\hat\theta_{oracle}-\abs{C})^2\big]\Big\}  < \infty\,.
\end{EQA}
\end{theorem}

\begin{remark}\label{RemAdapt}
The theorem implies that the rate of convergence for the RMSE (root mean-squared error) of the estimator \(\hat\theta_{oracle}\)
 is  \(\lambda^{-(d+3)/(2d+2)}\). In Theorem~\ref{lower_bound_theorem} below,  we prove that the lower bound on
the minimax risk in the PPP model is of the same order implying that
the rate is minimax-optimal. Even more,
the oracle estimator is \emph{adaptive} in the sense, that its rate is faster if the missing volume decays faster. In particular, for polytopes $C$ it is
shown in \cite{barany1988convex} and independently in \cite{dwyer1988convex} that
\(\E[ |C \setminus \Ch|] \thicksim \lambda ^{-1} (\log (\lambda|C|))^{d-1}, \)
which implies a faster (almost parametric) rate of convergence for the RMSE
of the oracle estimator.
\end{remark}

\begin{proof}
The unbiasedness follows immediately from the definition \eqref{sbvcbxcvb}.
By the law of total variance, we obtain
\begin{EQA}
\Var(\hat\theta_{oracle}) &	= & \Var\Big(\frac{N }{\lambda}\Big) - \E \Bigl[ \Var\Big(\frac{N }{\lambda} \cond \Ch\Big)\Bigr]
							 = \frac{|C|}{\lambda} - \E \Bigl[ \Var\Big(\frac{N_\circ }{\lambda}\cond \Ch\Big)\Bigr] \\
					&=&  \frac{|C|}{\lambda} - \E \Big[\frac{\lambda_\circ}{\lambda^2} \Big] = \frac{1}{\lambda} \E[|C \setminus \Ch|]\,.
\label{bcbnfwe}
\end{EQA}
Proposition~\ref{PropSuffCompl} below affirms that the convex hull \(\Ch\) is not only a sufficient, but also a complete statistic such that by  the Lehmann-Scheff\'e theorem, the estimator \(\hat\theta_{oracle}\) has the UMVU property.

Finally, we bound the expectation of the missing volume \( |C \setminus \Ch|\) by Poissonisation, i.e. using that the
 convex hull \(\Ch\) in the PPP model conditionally on the event \(\{N = k\}\)
is distributed as the convex hull \(\Ch_k = \text{conv}\{X_1, ..., X_k\}\) in the model with
$k$ uniform observations on $C$, for which
the following  upper bound is known (e.g., \cite{barany1988convex}):
\begin{EQA}[c]
\label{thm_upper_bound}
\sup_{C \in \C, |C| >0}  \E\Big[\frac{|C \setminus \Ch_k|}{\abs{C}}\Big] = O\big(k^{-2/(d+1)}\big)\,.
\end{EQA}
Thus, it follows by a Poisson moment bound
\begin{EQA}
\sup_{C \in \C, |C| >0} \E\Big[\frac{|C \setminus \Ch|}{\abs{C}^{(d-1)/(d+1)}}\Big]
	&=& \sup_{C \in \C,|C| >0}\sum_{k=0}^\infty  \frac{\ex^{-\lambda\abs{C}} (\lambda\abs{C})^k}{\abs{C}^{-2/(d+1)}k!} \E\Big[\frac{|C \setminus \Ch_k|}{\abs{C}}\Big] \\
	&=&O\big(\lambda^{-2/(d+1)}\big) \,.
\label{EcbCh}
\end{EQA}
This bound, together with \eqref{bcbnfwe}, yields the assertion.
\end{proof}

The lower bound for the risk in the PPP framework can be derived from
 the lower bound in the uniform model with a fixed number of observations, see Thm. 6 in \cite{gayraud1997}.

\begin{theorem}
\label{lower_bound_theorem}
For estimating \(|C|\) in the PPP model with parameter class \(\C\), the following asymptotic  lower bound holds
\begin{EQA}[c]
\liminf_{\lambda\to \infty}  \lambda^{(d+3)/(d+1)}\inf_{ \hat{\theta}_\lambda} \sup_{C \in \C}\E_C [ ( | C| -\hat{\theta}_\lambda)^2] > 0 \,,
\label{lower_bound_statement}
\end{EQA}
where the infimum extends over all estimators \(\hat{\theta}_\lambda\) in the PPP model with intensity \(\lambda\).
\end{theorem}
\begin{proof}
We use that an estimator \(\hat{\theta}_\lambda\) in the PPP model is an estimator in the uniform model on the event \(\{N = n \}\).
Then, due to the lower bound in the uniform model in \cite{gayraud1997},
for a constant \(c> 0\) and for all \(n \in  \mathbb{N}\) there exists
a set \(C_n \in \C\) with \(|C_n| \sim 1\) such that
 for all \( k = 1,...,n\),
\begin{EQA}[c]
\E_{C_n} \big[ ( | C_n| -\hat{\theta}_\lambda)^2 \cond N = k \big] >  c n ^{-(d+3)/(d+1)}, \quad a.s.  \,
\end{EQA}
Then, in the PPP model for \(C = C_{ \lfloor  \lambda \rfloor}\) with  \(\lambda |C| \ge 1\),  we have
\begin{EQA}
\E_{C} [ ( | C| -\hat{\theta}_\lambda)^2] & = & \sum_{k \in \mathbb{N}} \E_{C} \big[ ( | C| -\hat{\theta}_\lambda)^2 \cond N = k\big] \, \P (N = k) \\
		&\ge &\sum_{ k \le  \lfloor  \lambda \rfloor} \E_{C} \big[ ( | C| -\hat{\theta}_\lambda)^2 \cond N = k\big] \, \P (N = k) \\
		&> & c\lfloor  \lambda \rfloor^{-(d+3)/(d+1)} \, \big(1 - \P ( N > \lfloor  \lambda \rfloor) \big)\\
		&\sim &  \lambda^{-(d+3)/(d+1)} \,,
\end{EQA}
applying Chernoff's inequality to \(N \sim \text{Poiss}(  \lambda |C| )\)
for the last line. Thus, the lower bound \eqref{lower_bound_statement} follows.
\end{proof}

\begin{proposition}\label{PropSuffCompl}
For known intensity $\lambda>0$, the convex hull \(\Ch = \text{conv}\{X_1, ..., X_N\}\) is a  complete statistic.

\end{proposition}

\begin{proof}
We need to show the implication
\begin{EQA}[c]
\forall C \in \C: \E_C\bigl[T(\Ch)\bigr] = 0 \implies  T(\Ch) = 0 \;\;\; \P_{\mathbf E}-a.s.
\end{EQA}
for any \(\B_{\C}\)-measurable function \(T: \C \to \R\).
From the likelihood  in \eqref{fPPeb} for $\lambda=\lambda_0$, we derive
\begin{EQA}
\E_C\bigl[T(\hat{C})\bigr] & = &\E_{\mathbf E}\Bigl[T(\hat{C}) \exp\bigl(\lambda\abs{{\mathbf E}\setminus C}\bigl){\bf 1}(\hat{C} \subseteq C )\Bigr] \,.
\end{EQA}
Since \(\exp(\lambda\abs{{\mathbf E}\setminus C})\) is deterministic,  \(\E_C\bigl[T(\Ch)\bigr] = 0\) for all $C\in{\C}$ implies
\begin{EQA}[c]
\forall C \in \C:\E_{\mathbf E}\bigl[T(\hat{C}){\bf 1}(\hat{C} \subseteq C) \bigr] = 0 \,.
\end{EQA}
For \(C \in \C \), define the family of convex subsets of \(C\) as \([C] = \{A\in \C |  A \subseteq C \}\) such that $\hat C\subseteq C \iff \hat C\in[C]$. Splitting \(T = T^{+} - T^{-}\)
with non-negative \(\B_{\C}\)-measurable functions \(T^{+}\) and \(T^{-}\), we infer that the measures
\(\mu^\pm(B)=\E_{\mathbf E}[T^{\pm}(\hat{C}){\bf 1}(\hat{C} \in B)] \), $B\in\B_{\C}$, agree on \(\{[C] \,|\, C \in \C \}\).

Note that the  brackets  \(\{[C] | C \in \C \}\) are  \(\cap\)-stable due to \([A] \cap [C] = [A \cap C]\) and
\(A \cap C \in \C\). If the $\sigma$-algebra $\Cs$ generated by $\{[C]\,|\,C\in\C\}$ contains $\B_{\C}$, the uniqueness theorem asserts that
  the measures $\mu^+,\mu^-$
 agree on all Borel sets in \( \B_\C\), in particular on  \(\{T > 0\} \) and \(\{T < 0\} \), which entails $\E_{\mathbf E}[T^+(\hat C)]=\E_{\mathbf E}[T^-(\hat C)]=0$.
Thus, in this case, \(T(\Ch) = 0\) holds  \(\P_{\mathbf E}\)-a.s.

It remains to show that $\Cs=\sigma([C],\,C\in \C)$ equals the Borel \(\sigma\)-algebra \(\B_\C\). This can be derived as a non-trivial consequence of Choquet's theorem, see Thm.~7.8
 in  \cite{molchanov2006}, but we propose a short self-contained proof here.
Let us define the family \(\langle C \rangle = \{B \in \C | C \subseteq B\}\) of convex sets containing $C$.
Then the closed Hausdorff ball with center \(C\) and radius $\eps>0$ has the representation
\begin{EQA}[c]
B_\eps(C) \eqdef \{A \in \C \,|\, d_H(A, C) \le \eps \} = \{A \in \C\, |\, U_{-\eps}(C) \subseteq A \subseteq U_\eps(C) \}\,,
\end{EQA}
with $U_\eps(C)=\{x\in{\mathbf E}\,|\,\text{dist}(x,C)\le \eps\}$, $U_{-\eps}(C)=\{x\in C\,|\,\text{dist}(x,{\mathbf E}\setminus C)\le \eps\}$. Noting that $U_\eps(C),U_{-\eps}(C)$ are closed and convex and thus in ${\mathbf C}$, we obtain
\begin{EQA}[c]
\label{BeC}
B_\eps(C) = \langle U_{-\eps}(C) \rangle \cap [U_{\eps}(C)]\,.
\end{EQA}
Since $(\C,d_H)$ is separable, our problem is reduced to proving that all angle sets $\langle C\rangle$ for $C\in{\mathbf C}$ are in $\Cs$. A further reduction is achieved by noting $\langle C\rangle=\bigcap_{x\in C}\langle x\rangle=\bigcap_{x\in C\cap\Q^d}\langle x\rangle$ setting $\langle x\rangle=\langle \{x\}\rangle$ for short such that it suffices to prove $\langle x\rangle\in\Cs$ for all $x\in{\mathbf E}$.

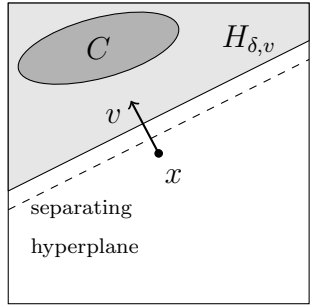
\begin{figure}[tp]
\centering
\begin{tikzpicture}
\draw (0,0) rectangle (4,4);
\draw [black, fill = gray!20!white] (0,1.5) -- (4,3.5) -- (4,4) -- (0,4) -- (0,1.5);
\draw [rotate around={15:(1.2,3.4)},black, fill = gray!60!white] (1.2,3.4) ellipse (1.1 and 0.4);
\draw [black, fill] (2,2) circle [radius=0.05];
\node at (2.2,1.7) {\(x\)};
\node at (3.2,3.5) {\(H_{\delta,v}\)};
\node at (1.2,3.4) {\(C\)};
\node at (1.4,2.5) {\(v\)};
\draw [->, thick] (2,2) -- (1.63,2.7);
\draw [black, dashed] (0,1.25) -- (4,3.25);
\node[text width=2cm] at (1.3,1) {\scriptsize separating \\ hyperplane};
\end{tikzpicture}
\caption{The construction used in the proof.}
\label{HB}
\end{figure}

Now, let \(x \in{\mathbf E}\) and \(C \in \C\) such that \(x \notin C\). Then, by the Hahn-Banach theorem, there are \(\delta > 0, v \in \R^d\) such that \(\langle v, c-x\rangle \ge \delta\) holds for all \(c \in C \).
By a density argument, we may choose $\delta\in\Q^+$ and $v\in\Q^d$. Denoting the corresponding hyperplane intersected with $\mathbf E$ by
$H_{\delta,v} = \{\xi \in {\mathbf E} \, | \, \langle v, \xi-x \rangle \ge \delta \}$, see Figure~\ref{HB},
we conclude
\begin{EQA}[c]
\langle x\rangle^\complement= \bigcup_{\delta \in \Q_{+} } \bigcup_{v \in \Q^d} \underbrace{[ H_{\delta,v}]}_{\in \Cs} \in \Cs \,.
\end{EQA}
Consequently, \(\langle x \rangle \in \Cs\) and thus $\B_\C\subseteq\Cs$ hold.
\end{proof}

\section{Unknown intensity \(\lambda\):  nearly unbiased estimation}
\label{unknown_intensity_section}

In case the intensity \(\lambda\) is unknown and the oracle estimator \(\hat\theta_{oracle}\) in \eqref{sbvcbxcvb}
is inaccessible, the  maximum-likelihood approach suggests to use \( N / |\Ch|\) as an estimator for $\lambda$ in  \eqref{fPPeb}.
This yields the \emph{plug-in} estimator for  the volume,
\begin{EQA}[c]
\hat\theta_{plugin} \eqdef |\Ch| +\frac{N_{\partial}}{N} |\Ch|\,.
\end{EQA}
In the unlikely event $N=\abs{\Ch}=0$, we define $\hat\theta_{plugin}=0$.
This estimator has a significant bias due to the following result, which is proved in the appendix.

\begin{lemma}
\label{lemma_bias_plug_in}
For the bias of the plug-in MLE estimator \(\hat\theta_{plugin}\), it follows with some universal constant \(c >0\)
\begin{EQA}[c]
\label{statement_bias_plug_in}
\abs{C} - \E[\hat\theta_{plugin}]  \ge  c \E[\abs{\Ch\setminus C}]^2\,, \quad \forall C \in \C\,.
\end{EQA}
\end{lemma}

The maximal bias over $C\in\C$ is thus at least of order $\lambda^{-4/(d+1)}$, which is worse than the minimax rate $\lambda^{-(d+3)/(2d+2)}$ for $d>5$. Yet, in the two-dimensional finite sample study of Section~\ref{numerical_study_section} below, its  performance is quite convincing.
We surmise that \(\hat\theta_{plugin}\) is rate-optimal for \(d \le 5\), but we leave that question aside
because the final estimator we propose  will be nearly unbiased and will satisfy an \emph{exact} oracle inequality. In particular, it is rate-optimal in any dimension. The new idea is to exploit that the number of interior points of $\Ch$ satisfies $N_\circ \cond \Ch \sim \text{Poiss}(\lambda_\circ)$, see \eqref{NccCs}.


\begin{remark}
There is no conditionally unbiased estimator for  \(\lambda_\circ^{-1}\) based on observing $N_\circ \cond \Ch \sim \text{Poiss}(\lambda_\circ)$ for $\lambda_\circ$ ranging over some open (non-empty) interval. Otherwise, an estimator \(\tilde{\mu}(N_\circ)\)
for $\lambda_\circ^{-1}$ would satisfy \(\E[\tilde{\mu}(N_\circ) | \Ch] =\lambda_\circ^{-1} \) implying
\begin{EQA}[c]
\sum_{k = 0}^{\infty}\frac{ \lambda_\circ^{k}}{k!} \tilde{\mu}(k) \ex^{-\lambda_\circ} = \lambda_\circ^{-1} \Rightarrow \sum_{k = 0}^{\infty}\frac{ \lambda_\circ^{k+1}}{k!} \tilde{\mu}(k) = \sum_{k = 0}^{\infty}\frac{ \lambda_\circ^{k}}{k!} \,.
\end{EQA}
The  coefficient for the constant term in the left and right power series would thus differ ($0$ versus $1$), in contradiction with the uniqueness theorem for power series.
\end{remark}

We provide an almost unbiased estimator for \(\lambda_\circ^{-1}\) by noting that the first jump time of a time-indexed
Poisson process
with intensity
$\nu$ is $\Exp(\nu)$-distributed and thus has expectation $\nu^{-1}$. Taking conditional expectation
of the first jump time with respect to the value of the Poisson process at time 1, we conclude that
\[ \hat\mu(N_\circ, \lambda_\circ) \eqdef \begin{cases} (N_\circ+1)^{-1},& \text{ for }N_\circ\ge 1,\\ 1+\lambda_\circ^{-1},&\text{ for } N_\circ=0\end{cases}
\]
satisfies $\E[\hat\mu(N_\circ,\lambda_\circ )|\Ch]=\lambda_\circ^{-1}$. Omitting the term $\lambda_\circ^{-1}$, depending on $\lambda_\circ$, in the unlikely case $N_\circ=0$, we
define our final estimator
\[\hat\theta \eqdef   | \Ch | +\frac{ N_{\partial}}{N_\circ + 1} | \Ch | \,.\]
For the proofs, we also define the  \emph{pseudo}-estimator
\[
\hat\theta_{pseudo} \eqdef  | \Ch | + | \Ch |N_{\partial}\Big(\frac{1}{N_\circ + 1}+ \frac{\ex^{-\lambda_\circ}}{\lambda_\circ}\Big)\,.
\]

\begin{theorem}
\label{estimator_bias}
 The pseudo-estimator \(\hat\theta_{pseudo}\) is unbiased
and the estimator \(\hat\theta\) is asymptotically unbiased in the sense that with constants \(c_1, c_2 > 0\) depending on $d$, $d>1$,
whenever $\lambda\abs{C}\ge 1$:
\begin{EQA}[c]
\label{CEt}
0 \le |C |  - \E [  \hat\theta ] \le c_1  |C| \exp\bigl(-c_2(\lambda|C|)^{(d-1)/(d+1)} \bigr)\,, \quad \forall C \in \C\,.
\end{EQA}
\end{theorem}

\begin{proof}
We have
\begin{EQA}[c]
\label{vxczvc214}
\E \Bigl[ \frac{1}{N_\circ+1}+ \frac{\ex^{-\lambda_\circ}}{\lambda_\circ}\cond \Ch \Bigr] = \ex^{-\lambda_\circ}\lambda_\circ^{-1} \Big(\sum_{k=0}^\infty \frac{\lambda_\circ^{k+1}}{(k+1)k!}+1\Big)=\lambda_\circ^{-1} \,,
\end{EQA}
which by $\abs{\Ch}\lambda_\circ^{-1}=\lambda^{-1}$ and $\E[\hat\theta_{oracle}]=\abs{C}$ implies  unbiasedness of \(\hat\theta_{pseudo}\).
Thus, it follows that
\[
 |C|-\E[\hat\theta] = \E \bigl[  | \Ch |N_{\partial} \ex^{-\lambda_\circ}\lambda_\circ^{-1}   \bigr] = \lambda^{-1} \E \bigl[ N_{\partial} \ex^{-\lambda |\Ch|}   \bigr] \,\,.
\]
We  exploit the deviation inequality from Thm.~1 in  \cite{Bru14c} 
and derive the bound for the exponential moment of the missing volume in the model with fixed number of points
\begin{EQA}[c]
\E[\exp{(\lambda|C \setminus \Ch_k|)}] \le b_1 \exp{(b_2\lambda|C| k^{-2/(d+1)})}\,, \quad k \ge 2\,,
\end{EQA}
for  positive constants \(b_1, b_2\), depending on the dimension according to \cite{Bru14c}. For the
cases \(k = 0,1\), we have the identity \(\E[\exp{(\lambda|C \setminus \Ch_k|)}] = \exp{(\lambda|C|)} \).
By Poissonisation, similarly to \eqref{EcbCh}, we derive
\begin{EQA}[c]
\label{sCPCbl}
\exp(-\lambda |C|) \E[\exp{(\lambda|C \setminus \Ch|)}] \le b_3 \exp{\big(-c_2(\lambda|C|)^{(d-1)/(d+1)}\big)}\,,
\end{EQA}
for  positive constants \(b_3, c_2\), depending on the dimension.
Hence,  using the Cauchy-Schwarz inequality and  the bound for the moments
 of the points on the convex hull,
\begin{EQA}[c]
\label{MomNumPoints}
\E [N_{\partial}^q] = O\bigl( (\lambda |C|)^{q(d-1)/(d+1)}\bigr) \,, \quad q \in \mathbb{N}\,,
\end{EQA}
 see e.g. Section~2.3.2 in \cite{brunel:tel-01066977},
we derive for a constant \(c_1 > 0\)
\begin{EQA}
\lambda^{-1} \E \bigl[ N_{\partial} \ex^{-\lambda |\Ch|}   \bigr]
			& \le & \lambda^{-1} \ex^{-\lambda |C|}   \E [N_{\partial}^2]^{1/2} \E[ \ex^{2\lambda |C \setminus \Ch|}]^{1/2}     \\
			& \le & c_1 \lambda^{-2/(d+1)}|C|^{(d-1)/(d+1)}
					\exp{\big(-c_2(\lambda|C|)^{(d-1)/(d+1)}\big)}\\
			& \le&  c_1 |C| \exp{\big(-c_2(\lambda|C|)^{(d-1)/(d+1)}\big)} \,.
\label{lexlC}
\end{EQA}

\end{proof}


The next step of the analysis is to  compare the variance of the {pseudo}-estimator \(\hat\theta_{pseudo} \)
with the variance of the oracle estimator \(\hat\theta_{oracle} \), which is UMVU.

\begin{theorem}
\label{estimator_variance}
The following oracle inequality holds with a universal constant $c>0$ and dimension-dependent constants $c_1,c_2 >0$ for all $C \in \C$  with $\lambda\abs{C}\ge 1$:
\begin{EQA}[c]
\Var(\hat\theta_{pseudo} )\le  (1 + c \alpha(\lambda, C))\Var(\hat\theta_{oracle})
						+  r(\lambda,C)\,,
\end{EQA}
where
\begin{align*}
\alpha(\lambda, C) &= \frac{1}{\abs{C}} \Bigl( \frac{1}{\lambda} + \frac{ \Var(\abs{C\setminus\Ch} )}{ \E[ \abs{C\setminus\Ch}]}
	  +   \E[ \abs{C\setminus\Ch}] \Bigr)\,,\\
r(\lambda,C) &=c_1 (\lambda |C |)^{2(d-1)/(d+1)} \exp{\big(-c_2(\lambda|C|)^{(d-1)/(d+1)}\big)}\,.
\end{align*}
\end{theorem}

\begin{proof}
By the law of total variance, we obtain
\begin{EQA}
\Var(\hat\theta_{pseudo} ) &=& \Var\bigl(\E [\hat\theta_{pseudo}| \Ch]\bigr)
								+ \E \bigl[ \Var(\hat\theta_{pseudo}| \Ch) \bigr] \\
						&=&  \Var(\hat\theta_{oracle})
								+ \E \Bigl[ ( N_{\partial} | \Ch |)^2 \Var\Big( \frac{ 1}{N_\circ + 1}
								| \Ch\Big) \Bigr]\,.
\end{EQA}
In view of \(N_\circ\,|\,\Ch  \sim \text{Poiss}(\lambda_\circ)\), a power series expansion gives
\[ \E[(N_\circ+1)^{-2}|\,\Ch]=\lambda_\circ^{-1}\ex^{-\lambda_\circ}\int_0^{\lambda_\circ}(\ex^{t}-1)/t\,dt\,.\]
The conditional variance can for $\lambda_\circ\to\infty$ thus be bounded by
\begin{align*}
\Var((1+N_\circ)^{-1}|\,\Ch) &\le \lambda_\circ^{-1}\ex^{-\lambda_\circ}\int_{\lambda_\circ/2}^{\lambda_\circ}\ex^{t}/t\,dt-(\lambda_\circ)^{-2}+O(\ex^{-\lambda_\circ/4})\\
&=(\lambda_\circ)^{-1}\int_{0}^{\lambda_\circ/2} \ex^{-s}\Big(\frac1{\lambda_\circ-s}-\frac1{\lambda_\circ} \Big)\,ds+O(\ex^{-\lambda_\circ/4})\\
&=\lambda_\circ^{-3}(1+o(1))\,,
\end{align*}
where we have used $(\lambda_\circ-s)^{-1}-\lambda_\circ^{-1}=s\lambda_\circ^{-1}(\lambda_\circ-s)^{-1}$, $\int_0^\infty s\ex^{-s}ds=1$ and dominated convergence. Thanks to $(N_\circ+1)^{-1}\in[0,1]$ we conclude for some constant $c\ge 1$
\[ \Var((1+N_\circ)^{-1}|\,\Ch)\le c(1\wedge \lambda_\circ^{-3}).\]
Consequently, we have
\begin{EQA}
\Var(\hat\theta_{pseudo} )  &\le&  \Var(\hat\theta_{oracle})
								+ \E \bigl[ ( N_{\partial} | \Ch |)^2 c (1 \wedge (\lambda |\Ch|)^{-3}) \bigr] \\
						    &=&  \Var(\hat\theta_{oracle})
								+ c \E \bigl[ ( N_{\partial} | \Ch |)^2  \wedge \lambda^{-3} (N_{\partial})^2 |\Ch|^{-1} \bigr]\,,
\end{EQA}
and with \eqref{bcbnfwe}
\begin{EQA}
\frac{\Var(\hat\theta_{pseudo} )}{\Var(\hat\theta_{oracle})} & \le & 1 + c \frac{ \E \bigl[ ( N_{\partial} \lambda | \Ch |)^2  \wedge  (N_{\partial})^2 (\lambda |\Ch|)^{-1} \bigr]}{ \lambda \E [|C \setminus \Ch|]} \\
& = &1 + c \frac{ \E \bigl[  (N_{\partial})^2  \bigl( ( \lambda | \Ch |)^2  \wedge  (\lambda |\Ch|)^{-1}\bigr)  \bigr]}{ \E[ N_{\partial}] }\,.
\label{fdfsadf1}
\end{EQA}
Define the `good' event \(\mathcal{G} = \{|\Ch| > |C|/2\}\), on which
\(\bigl( (\lambda  | \Ch |)^2  \wedge  (\lambda |\Ch|)^{-1}\bigr)  \le 2(\lambda |C|)^{-1}\).
On the complement  \(\mathcal{G}^c\), we infer from $A^2\wedge A^{-1}\le 1$ for $A>0$
\begin{EQA}
 &&\E \Bigl[  (N_{\partial})^2  \bigl( ( \lambda | \Ch |)^2  \wedge  (\lambda |\Ch|)^{-1}\bigr) {\bf 1}_{\mathcal{G}^c}  \Bigr] \le
		\E \bigl[  N_{\partial}^2    {\bf 1}_{\mathcal{G}^c}  \bigr] \\
		&&\quad \le \E[N_{\partial}^4]^{1/2} \P(|C\setminus \Ch| \ge |C|/2)^{1/2} \\
		&&\quad \le c_1 (\lambda |C |)^{2(d-1)/(d+1)} \exp{\big(-c_2(\lambda|C|)^{(d-1)/(d+1)}\big)}\,,
\label{EBNchlC}
\end{EQA}
for some positive constant \(c_1\) and \(c_2\), using \eqref{sCPCbl} and
\eqref{MomNumPoints}.
It remains to estimate the upper bound \eqref{fdfsadf1} on $\mathcal{G}$
\begin{EQA}[c]
\frac{2c}{\lambda\abs{C}}\frac{ \E  [N_{\partial}^2] }{ \E [N_{\partial}] } = \frac{2c}{\lambda\abs{C}}\Big(\frac{ \Var  (N_{\partial}) }{ \E [N_{\partial}] } + \E[ N_{\partial}]\Big) \,.
\label{laCfCf}
\end{EQA}
Using the identity (17) in \cite{reitzner2015} for the factorial moments for the number of vertices \(N_{\partial}\),
 we derive \(\Var  (N_{\partial})  \le \lambda^2 \Var(\abs{C\setminus\Ch})
+  \lambda\E[ \abs{C\setminus\Ch}]  \) in view of \(  \E [N_{\partial}] = \lambda \E[ \abs{C\setminus\Ch}]\). Thus, \eqref{laCfCf}
is bounded by
\begin{EQA}[c]
\frac{2c}{\lambda\abs{C}}\frac{ \E  [N_{\partial}^2] }{ \E [N_{\partial}] }  \le
	  \frac{2c}{\abs{C}} \Bigl( \frac{1}{\lambda} + \frac{ \Var(\abs{C\setminus\Ch} )}{ \E[ \abs{C\setminus\Ch}]}
	  +   \E[ \abs{C\setminus\Ch}] \Bigr)\,,
\end{EQA}
which yields the assertion.
\end{proof}

As a result, we obtain an \emph{oracle inequality} for the estimator \(\hat\theta\).

\begin{theorem}
\label{new_estimator_risk}
It follows for the risk of the estimator \(\hat\theta\)
for all $C \in \C$ whenever $\lambda\abs{C}\ge 1$:
\begin{EQA}[c]
\E[(\hat\theta - |C|)^2]^{1/2} \le  (1+c \alpha(\lambda, C))\E[(\hat\theta_{oracle}-\abs{C})^2]^{1/2} +r(\lambda,C)\,,
\end{EQA}
 with constant \(c > 0\) and \(\alpha(\lambda, C),r(\lambda,C)\) from Theorem \ref{estimator_variance}. For any $C\in\C$ and $\lambda>0$ we have $\alpha(\lambda,C)\le 1+\frac{1}{\lambda\abs{C}}$.
\end{theorem}

\begin{proof}
In view of $\lambda_\circ=\lambda\abs{\Ch}$, we have \(\hat\theta  = \hat\theta_{pseudo} - \lambda^{-1}N_{\partial}\ex^{-\lambda\abs{\Ch}}\) and we derive as in \eqref{lexlC} and \eqref{EBNchlC} with some constants $c_1,c_2>0$
\begin{align*}
\E[(\hat\theta -\hat\theta_{pseudo})^2] &\le \lambda^{-2}\E[N_{\partial}^4]^{1/2}\E[\ex^{-4\lambda\abs{\Ch}}]^{1/2}
\le c_1^2 \exp{\big(-2c_2 (\lambda\abs{C})^{(d-1)/(d+1)}\big)}.
\end{align*}
To establish the oracle inequality, we apply the triangle inequality in $L^2$-norm together with Theorems \ref{ThmOracle} and \ref{estimator_variance}.

The universal bound on $\alpha(\lambda,C)$ follows from the rough bound $\E[\abs{C\setminus\Ch}^2]\le \abs{C}\E[\abs{C\setminus\Ch}]$.
\end{proof}

Note that the remainder term $r(\lambda,C)$ is exponentially small in $\lambda\abs{C}$. Therefore,
an immediate implication of Theorem~\ref{new_estimator_risk} is that asymptotically our estimator \(\hat\theta\)
 is
minimax rate-optimal in all dimensions, where the lower bound is proved in the next section. Yet, even more is true: the oracle inequality  is in all well studied cases \emph{exact} in the sense that $\alpha(\lambda,C)\to 0$ holds for $\lambda\to\infty$ such that the
 the UMVU risk of $\hat\theta_{oracle}$ is attained asymptotically.

\begin{lemma}
\label{LemOracle}
We have tighter bounds on \(\alpha(\lambda, C)\)  from Theorem \ref{estimator_variance} in the following cases:
\begin{enumerate}
\item for \(d = 1,2\) and  \( C \in \C\) arbitrary: \( \alpha(\lambda, C) \lesssim (\lambda |C|)^{-2/(d+1)}\),
\item for \( d \ge 2\), \(C\)  with $C^2$-boundary of positive curvature: \( \alpha(\lambda, C) \lesssim  (\lambda |C|)^{-2/(d+1)}\),
\item for \( d \ge 2\) and \(C\)  a polytope: \( \alpha(\lambda, C) \lesssim \lambda^{-1} (\log(\lambda|C|))^{d-1}\).
\end{enumerate}
\end{lemma}
\begin{proof}
Let us restrict to \(|C| = 1\), the case of general volume follows by rescaling. In view of  the expectation upper bound \eqref{EcbCh}, the main issue is to bound
\( \Var(\abs{C\setminus\Ch} ) /  \E[ \abs{C\setminus\Ch}]\)  uniformly.
Case (1) follows from \cite{pardon2011}, where \( \lambda \Var(\abs{C\setminus\Ch}) \thicksim \E[ \abs{C\setminus\Ch}]  \) is established.

For case (2) with smooth boundary, the upper bound for the variance,
\(\Var(\abs{C\setminus\Ch} ) \lesssim \lambda^{-(d+3)/(d+1)}\), was obtained in \cite{reitzner2005central}, while the lower bound
for the first moment, \( \E[ \abs{C\setminus\Ch}] \gtrsim \lambda^{-2/(d+1)}\), is due to \cite{schutt1994random}.

For the case (3) of polytopes,   the upper bound  \(\Var(\abs{C\setminus\Ch} ) \lesssim \lambda^{-2}(\log \lambda)^{d-1}\) was obtained in \cite{barany2010variance},
while the lower bound for the first moment, \( \E[ \abs{C\setminus\Ch}] \gtrsim \lambda^{-1}(\log \lambda)^{d-1}\), was proved in \cite{barany1988convex}. The expectation upper bound from Remark \ref{RemAdapt} thus yields the result.
\end{proof}

One could conjecture that \( \lambda \Var(\abs{C\setminus\Ch}) \thicksim \E[ \abs{C\setminus\Ch}]  \)
 holds universally for all convex sets in arbitrary dimensions and thus that the oracle inequality is always exact.
Proving such a universal bound is a challenging open problem in stochastic geometry, strongly connected to the discussion on universal variance asymptotics in terms of the floating body by \cite{barany2010variance}.



\section{Finite sample behaviour and dilated hull estimator}
\label{numerical_study_section}

\begin{figure}[tp]
\centering
  \includegraphics[width=0.8\linewidth]{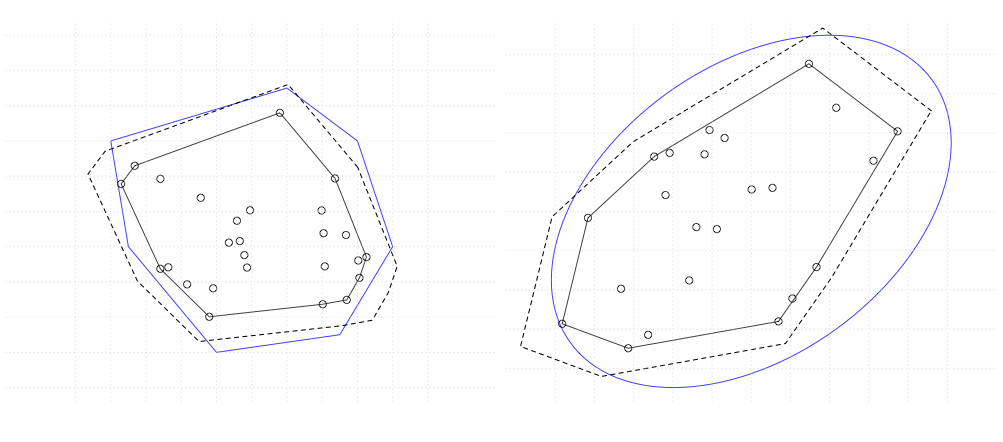}\\[-5mm]
  \caption[]{The two  convex sets (blue), observations (points), their convex hulls (black lines)
and dilated convex hulls (black dashed lines).}
\label{polygon_ellipse_figures}
\end{figure}

In this section, we demonstrate the performance of the main estimator \(\hat{\theta}\) numerically and compare it to other estimators
including the naive estimator \(|\Ch|\),  the naive oracle estimator \(N / \lambda\),  the UMVU oracle estimator \(\hat{\theta}_{oracle}\) and the plug-in MLE estimator \(\hat\theta_{plugin} = |\Ch| (1+ N_{\partial}/ N)\).
The main competitor from the literature is a  rate-optimal  estimator proposed in \cite{gayraud1997}. In their construction, the whole sample
is divided into three equal parts \(X\), \(X^\prime\) and \(X^{\prime\prime}\) of sizes \(N^\star  \)
 (without loss of generality \(N^\star\in\mathbb{N}\)) and the estimator is given by
\begin{EQA}[c]
\hat\theta_{G} = |\Ch| + \frac{|\Ch^{\prime\prime}|}{N^\star} \sum_{i=1}^{N^\star } {\bf 1} (X^\prime_i \notin \Ch)\,,
\end{EQA}
where \(\Ch^{\prime\prime}\)  is the convex hull of the third sample \(X^{\prime\prime}\).
The data points are simulated for  two convex sets: an ellipse and a polygon; see Figure~\ref{polygon_ellipse_figures}.

\begin{figure}[tp]
  \includegraphics[width=\linewidth]{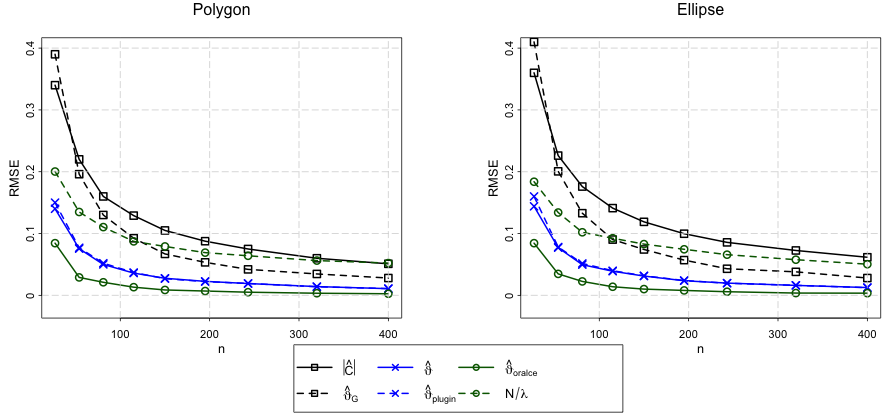}
  \caption[]{Monte Carlo RMSE estimates for the studied estimators for the volume of two convex sets: a polygon and an ellipse.}
\protect\label{rmse_estimates_plot}
\end{figure}

The RMSE  estimate normalised by the area of the true set  is based on \(M = 500\) Monte
Carlo iterations in each case.
 The results
of the simulations  are depicted in
 Figure~\ref{rmse_estimates_plot} where $n=\lambda\abs{C}$ denotes the expected total number of points. The worst convergence rate of $N/\lambda$ is clearly visible.
More importantly, we  see that the RMSE of $\hat\theta$ approaches the oracle risk for larger $n$ (i.e. $\lambda$) as the oracle inequality predicts.
It is also conspicuous that in the studied cases the plug-in estimator \(\hat\theta_{plugin}\) and the estimator
\(\hat\theta\) perform rather similarly. This is explained by the fact that the number of points \(N_{\partial}\) on the convex hull
increases with a moderate speed  in the two-dimensional case, \(\E[N_{\partial}] = O(\lambda^{1/3})\), which results in a
small difference between the multiplication factors \(N_{\partial} / N\) and \(N_{\partial} / (N_\circ + 1) \).
The simulations in two dimensions were implemented using the R package ``spatstat'' by \cite{spatstat}. To illustrate the
sub-optimality of the plug-in estimator \(\hat\theta_{plugin}\) in high dimensions,
we provide results of numerical simulations in dimensions \(d = 3,4,5,6\) for the case when the true set
\(C \) is a unit cube \(C = [0,1]^d\), see Figure~\ref{d3456}.
The simulations were implemented using the R package ``geometry'' by \cite{geometry}.

As an application of the obtained results, we propose a new estimator for the convex set itself:
\begin{EQA}
\tilde{C} & \eqdef & \Big\{\hat x_0+\Bigl(\frac{\hat\theta}{|\Ch|}\Bigr)^{1/d} (x-\hat x_0)\,\Big|\,x\in\Ch\Big\}  \\
		&=& \Big\{\hat x_0+\Bigl(\frac{N+1}{N_\circ+1}\Bigr)^{1/d}  (x-\hat x_0)\,\Big|\,x\in\Ch\Big\} \,,
\end{EQA}
which is just the dilation of the convex hull $\Ch$ from its barycentre $\hat x_0$,
 see the dashed polygons in Figure~\ref{polygon_ellipse_figures}. Since the convex hull is a sufficient statistic (for known $\lambda$), the points in its interior do not bear any information on the shape of $C$ itself such that the barycentre is a reasonable choice. There are, of course, other enlargements of the convex hull conceivable like  \(\argmin_{B \in \C, |B| = \hat\theta} d_H(B,\Ch)\), the convex set closest (in Hausdorff distance) to $\Ch$ with volume $\hat\theta$. The intuition behind these estimators is based on the observation
 that once the volume of the true set is known, we can estimate the set itself faster (in the constant), and
\(\hat\theta\) is a reasonable substitute for the true volume due to its fast rate of convergence.

\begin{figure}[tp]
\centering
  \includegraphics[width=\linewidth]{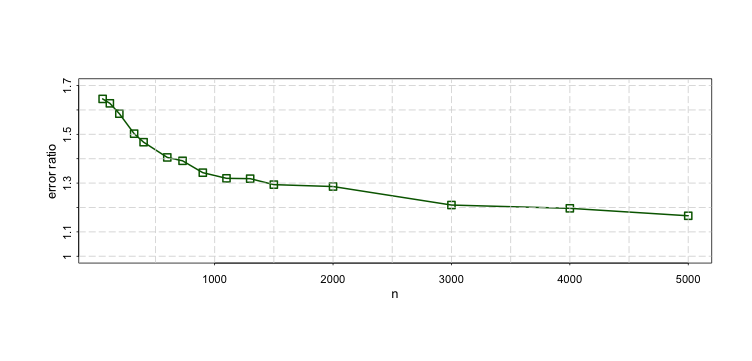}\\[-1cm]
\caption{Monte Carlo error ratio for the convex hull and its dilation when the true set is a polygon.}\label{error_ratio_plot}
\end{figure}

A detailed analysis is not pursued here, but in a small simulation study we investigate the behaviour of the new dilated hull
estimator for the above polygon.
The error ratio \(\E[ \abs{C \Delta \Ch}] / \E [\abs{C \Delta \tilde{C}} ]  \) in terms of
 the symmetric difference $A\Delta B= ( A \setminus B) \cup  (B \setminus A)$ is approximated in \(M = 500\) Monte Carlo iterations and shown in Figure~\ref{error_ratio_plot}. It turns out that the dilation significantly improves the convex hull as an estimator for \(C\), especially for a small number of observations.


\section{Appendix}
\subsection{Proof of Theorem~\ref{LemMeasCond}}

The proof is split into several statements, which might be of interest on their own.

\begin{lemma}
\label{LemMeas}
The random variable \(\N({\Kh})\) is measurable with respect to \(\F_\Kh\) for any stopping set \(\Kh\).
\end{lemma}
\begin{proof}
The proof is just a generalisation of the analogous statement for time-indexed stochastic processes, see
e.g. Proposition 2.18 in \cite{karatzas2012brownian}. For this, the notions are extended to the partial order $\subseteq$ and
then the right-continuity of $(\N(K), K\in\K)$ (with respect to inclusion) implies its progressive measurability and thus in turn
the measurability of \(\N({\Kh})\).
\end{proof}

Next, observe that the set-indexed process \((\N(K), K \in \K)\) has independent increments, i.e. for
\(K_1,\ldots,K_m \in \K\) with \(K_i \subseteq K_{i+1}\), $i=1,\ldots,m-1$, the random variables \(\N(K_{i+1}) - \N(K_i)=\N(K_{i+1}\setminus K_i)\)are independent (by the independence of the PPP on disjoint sets).
In fact, we show in Proposition~\ref{psm} that the process \(\N\) is even a strong Markov process. In addition,
Proposition~\ref{psm} yields \eqref{NccCs} using
that the closed complement \(\hat\Kh=\bar{{\hat{C}}^c}\) of the convex hull is a stopping set.

\begin{proposition}
\label{psm}
The set-indexed process \((\N(K), K\in\K)\) is strong Markov  at every stopping set \(\Kh\). More precisely,
conditionally on \(\F_{\Kh}\) the process \((\N(K\setminus \Kh),K\in\K)\) is  a Poisson point process with intensity $\lambda$ on $\Kh^c$.
In particular, \(\N(K\setminus \Kh ) \cond \F_{\Kh}  \sim \text{Poiss}(\lambda|K\setminus \Kh|)\) holds for all  \(K \in \K\).
\end{proposition}

\begin{remark}
The fact that  the increments  \(\N(\Kh \cup K) - \N(\Kh)\) are independent of \(\F_{\Kh}\)
 can be derived from a general theorem about the strong Markov property for random fields
in Thm. 4 in \cite{rozanov1982markov}.
See also \cite{zuyev2006strong}
for a discussion of the strong Markov property and its applications in stochastic geometry.
These statements, however, do not provide
a distributional characterisation of the increments of the process.
\end{remark}

\begin{proof}
A set-indexed, \((\F_K)\)-adapted integrable process
\((X_K, K \in \K)\) is called a martingale if  \(\E[X_B| \F_A] = X_A\) holds for any \(A, B \in \K\) with \(A \subseteq B\).
By the independence of increments, the process \(M_K \eqdef \N(K) - \lambda|K|\), \(K\in\K\), is clearly a martingale with respect to its natural filtration \((\F_K,K\in\K)\).
Then also the process
\[ \tilde{M}_K \eqdef M_{K\cup \Kh}-M_{\Kh}=\N(K\setminus\Kh)-\lambda\abs{K\setminus\Kh}\]
is a martingale with respect to the filtration \(\tilde{\F}_{K} \eqdef \F_K \vee \F_{\Kh}  = \F_{K\cup\Kh}\) because
 for  \(K_1, K_2 \in \K\) with \(K_1 \subseteq K_2\) the optional sampling theorem (see e.g. \cite{zuyev1999stopping}) yields
\[ \E[\tilde{M}_{K_2}| \tilde{\F}_{K_1}] = \E[M_{K_2\cup \Kh}-M_{\Kh}| \F_{K_1\cup \Kh}] = M_{K_1\cup \Kh}-M_{\Kh}=\tilde{M}_{K_1},\]
noting that $K_1\cup\Kh$ is again a stopping set.

This implies that \(\lambda|K\setminus\Kh|\), conditionally on \(\Kh\), is the deterministic compensator of the process \(\tilde{N}_K  = \N(K\setminus \Kh)\).
Then, due to the martingale characterisation  of the set-indexed Poisson process,
see Thm. 3.1 in \cite{ivanoff1994martingale} (analogue of  Watanabe's characterisation for the Poisson process),
the process \(\tilde{N}_K\), conditionally on \(\F_{\Kh}\), is a Poisson point process with mean measure \(\tilde\mu(A)= \lambda|A \cap \Kh^c| \).
\end{proof}

The last statement of Theorem~\ref{LemMeasCond},
 that  \(\F_{\hat\Kh} = \sigma(\Ch) \)
 is shown next. It can be seen as a generalisation of the
 interesting fact that
for a time-indexed Poisson process the sigma-algebra \(\sigma(\tau)\) associated
with the first jump time \(\tau\) coincides with the sigma-algebra
of \(\tau\)-history \(\F_\tau\).

\begin{lemma} The sigma-algebra \(\sigma(\Ch)\)
coincides with  the sigma-algebra \( \F_{\hat\Kh}\) of \(\hat\Kh\)-history, i.e. \(\sigma(\Ch) = \F_{\hat\Kh}\).
\end{lemma}

\begin{proof}
Since $\hat\Kh$ is $\F_{\hat\Kh}$-measurable by Lemma~1 in \cite{zuyev1999stopping} and $\Ch=\overline{\hat\Kh^c}$, it is evident that \(\sigma(\Ch) \subseteq \F_{\hat\Kh}\). The other direction is  more involved.
We use that the sigma-algebra \( \F_{\hat\Kh}\)  coincides with the sigma-algebra
\(\sigma(\{\N(\hat\Kh\cap K), K \in \K\})\) generated by the  process stopped at $\hat\Kh$. This statement
can be derived from Thm. 6, Ch. 1 in \cite{shiryaev2007optimal}. Note that their assumption (1.11)  is satisfied in our case,
because for all \(K \in \K\) and \(\omega \in \Omega \) there is \(\omega^\prime\) such that
 \(\N(U\cap K,\omega) = \N(U,\omega^\prime) \) for all \(U \in \K\), which simply says that observing points in \(K\in \K\)
 there might be no points outside \(K\).
Finally, observe that by definition of the convex hull $\N(\bar{\Ch^c}\cap K)=\N((\partial\Ch)\cap K)$. 
 Modulo null sets, $\N((\partial\Ch)\cap K)$ counts the number of vertices of $\hat C$ in $K$ and is thus $\sigma(\hat C)$-measurable.
\end{proof}

\begin{proof}[Proof of Lemma~\ref{lemma_bias_plug_in}]
Using that the bias of the oracle estimator \(\hat\theta =  | \Ch | +  N_{\partial}/ (N_\circ + 1) | \Ch | \) is exponentially small, it remains
to compare its expectation with the expectation of the plug-in estimator \(\hat\theta_{plugin}\) to show \eqref{statement_bias_plug_in}:
\begin{EQA}
\E [\hat\theta  - \hat\theta_{plugin}] & =   & \E \biggl[ |\Ch| (\frac{N_{\partial}}{N_\circ + 1} - \frac{N_{\partial}}{N}) \biggr]
								=\E \biggl[|\Ch| \frac{N_{\partial}^2 - N_{\partial}}{(N_\circ + 1) (N_\circ + N_{\partial})}  \biggr] \\
						       &\ge & \frac {d}{d+1}\E \biggl[ \frac{|\Ch|N_{\partial}^2 }{(N_\circ + 1)2\lambda\abs{C}}{\bf 1}(N\le 2\lambda\abs{C})  \biggr]	 \,,
\label{Etvtvpl}
\end{EQA}
where in the last line we have used $\abs{\Ch}>0$ only if $N_{\partial}\ge d+1$ and in this case $N_{\partial}^2-N_{\partial}\ge \frac{d}{d+1}N_{\partial}^2$.
Using $\E[(N_\circ+1)^{-1}\,|\,\hat C]=\lambda_\circ^{-1}(1-\ex^{-\lambda_\circ})$ from above, we obtain after writing ${\bf 1}(N\le 2\lambda\abs{C})=1-{\bf 1}(N> 2\lambda\abs{C})$
\begin{align*}
\E [\hat\theta  - \hat\theta_{plugin}] &\ge \frac {d}{d+1}\Big(\E \biggl[ \frac{N_{\partial}^2 |\Ch|(1-\ex^{-\lambda_\circ})}{2\lambda_\circ \lambda\abs{C}}\biggr]
-\E \biggl[ \frac{N_{\partial}^2 |\Ch|}{2\lambda\abs{C}}{\bf 1}(N> 2\lambda\abs{C})  \biggr]\Big)\\
&\ge \frac {d}{d+1}\Big(\frac{\E[N_{\partial}^2(1-\ex^{-\lambda_\circ})]}{2\lambda^2\abs{C}}
-\frac{\E \bigl[ N^2{\bf 1}(N> 2\lambda\abs{C})  \bigr]}{2\lambda}\Big).
\end{align*}
By Cauchy-Schwarz inequality and large deviations similarly to \eqref{lexlC}, the first term is bounded from below by a constant multiple of $\E[\abs{C\setminus\Ch}]^2/\abs{C}$ in view of \(\E[N_{\partial}^2] \ge \lambda^2 \E[\abs{C\setminus\Ch}]^2 \),  see e.g. Section~2.3.2 in \cite{brunel:tel-01066977}. Because of $N\sim\Poiss(\lambda\abs{C})$, the second term is of order $\lambda\abs{C}^2\ex^{-\lambda\abs{C}}$ and thus asymptotically of much smaller order.
\end{proof}







\bibliographystyle{economet}



\bibliography{ref}

\begin{figure}[htbp]
  \includegraphics[width=\linewidth]{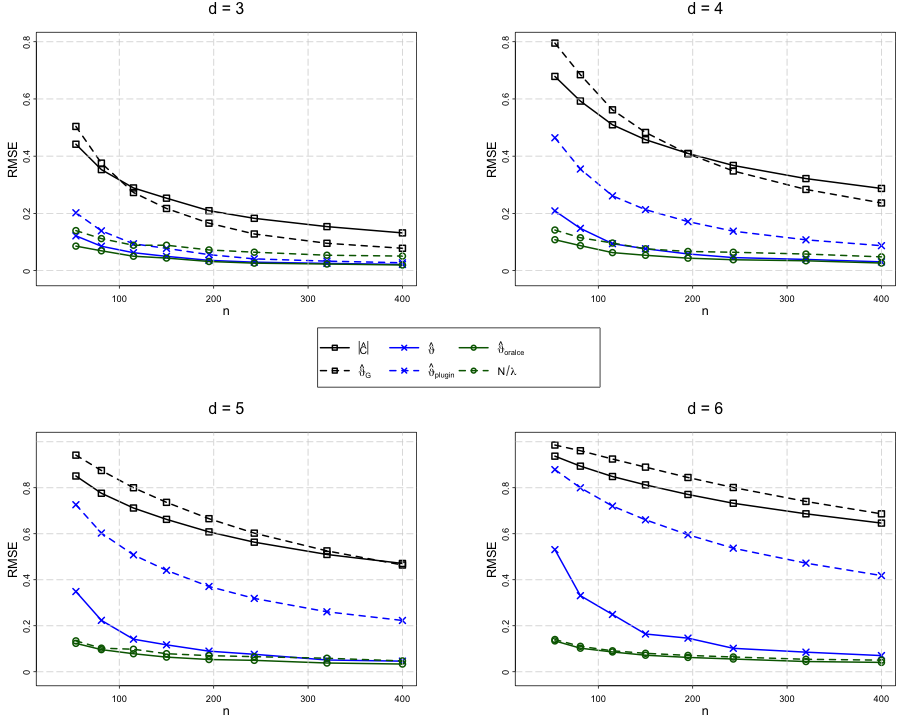}
  \caption[]{Monte Carlo RMSE estimates for the studied estimators for the volume of the unit cube \(C = [0,1]^d\) in dimensions  \(d = 3,4,5,6\) .}
\protect\label{d3456}
\end{figure}

\end{document}